\documentclass[a4paper,10pt,reqno]{amsart}
\usepackage{graphicx}
\usepackage{subfigure}
\usepackage[utf8]{inputenc}
\usepackage{amsmath,amssymb}
\usepackage{verbatim}
\usepackage[colorlinks=true]{hyperref}
\usepackage[usenames,dvipsnames]{color}
\usepackage{xspace}
\usepackage{bm}
\usepackage{nicefrac}

\newtheorem{thm}{Theorem}[section]
\newtheorem{prop}[thm]{Proposition}
\newtheorem{lemma}[thm]{Lemma}
\newtheorem*{lemma*}{Lemma}
\newtheorem{cor}[thm]{Corollary}

\theoremstyle{definition}
\newtheorem{defi}[thm]{Definition}
\newtheorem{rmk}[thm]{Remark}

\newcommand{\beginfig}{\begin{figure}[htbp]\centering}

\def\coleq{\mathrel{\mathop:}=}

\def\sp{\operatorname{span}}
\def\Id{\operatorname{Id}}
\newcommand{\ket}[1]{|#1\rangle}
\newcommand{\SymV}{\mathrm{Sym}^{\ell}V}
\newcommand{\res}{\operatornamewithlimits{Res}}

\title[The higher spin 6-vertex model and Macdonald polynomials]{The higher spin generalization of the 6-vertex model with domain wall boundary conditions and Macdonald polynomials}
\author{Tiago Fonseca}
\address{Centre de Recherches Mathématiques\\Université de Montréal}
\curraddr{LAPTh and CNRS\\9 chemin de Bellevue\\BP 110\\74941 Annecy-le-Vieux Cedex\\France}
\email{tiago.dinis.da.fonseca\,@\,sapo.pt}

\author{Ferenc Balogh}
\address{Concordia University, Centre de Recherches Mathématiques}
\curraddr{SISSA\\via Bonomea, 265\\34136 Trieste\\Italia} 
\email{fbalogh\,@\,sissa.it}

\begin{document}

\begin{abstract}
The determinantal form of the partition function of the $6$-vertex model with domain wall boundary conditions was given by Izergin.
It is known that for a special value of the crossing parameter the partition function reduces to a Schur polynomial.

Caradoc, Foda and Kitanine computed the partition function of the higher spin generalization of the $6$-vertex model.
In the present work it is shown that for a special value of the crossing parameter, referred to as the combinatorial point, the partition function reduces to a Macdonald polynomial.

\end{abstract}

\maketitle

\section*{Introduction}
The $6$-vertex model with domain wall boundary conditions was introduced by Korepin in \cite{Kor}, where the partition function was shown to satisfy certain recursion relations.
In \cite{Iz-6V}, Izergin solved Korepin's recursion relations with the result being a determinantal formula.

An interesting feature of this model is that it has multiple combinatorial interpretations
since its configurations are in bijection with several combinatorial objects, such as alternating sign matrices, fully packed loops, and states of a square ice model \cite{Bressoud, MRR-ASM, MRR-TSSCPP, dG-review}. 
For example, Kuperberg~\cite{Kup-ASM} used this model to compute the number of alternating sign matrices.

\medskip

The $6$-vertex model is an integrable model, meaning that it possesses an $R$--matrix which satisfies the Yang--Baxter equation. It is this algebraic structure that allows the partition function to be calculated explicitly.

To each edge of the $6$-vertex model a $\nicefrac{1}{2}$ spin is associated, since the corresponding $R$-matrix in intimately connected with the representation theory of the algebra $sl_2$. It is natural to search for generalizations of this model, where the representation of the underlying algebra or even the algebra itself are replaced (for example, by $sl_r$).
The idea is to construct an $R$--matrix for the model, that satisfies the Yang--Baxter equation (see, e.g.~\cite{PM-19vertex}).

If we restrict ourselves to studying general irreducible representations of $sl_2$, there is a systematic way of constructing the $R$--matrix, referred to as \emph{fusion}~\cite{Fusion-KRS,Fusion-Reshetikhin}.
The idea comes from the simple fact that in the representation theory of $sl_2$ we can build the spin $\nicefrac{\ell}{2}$ representation through the fusion of $\ell$ spin $\nicefrac{1}{2}$ representations.
This was achieved, for the 6-vertex model with domain-wall boundary conditions, in the work of Caradoc, Foda and Kitanine~\cite{CFKitanine-6V}, which serves as the basis of our paper.

\medskip

The partition function of the $6$-vertex model, as defined in Section~\ref{sec:def-6v}, is a multivariate polynomial which is symmetric in two separate sets of variables (known as \emph{spectral parameters}). It also depends on an extra parameter, normally denoted by $q$, and referred to as the \emph{crossing parameter}.

If we set $q=\exp(2\pi i/3)$, known as the \emph{combinatorial point}, the partition function becomes symmetric in the two sets of variables as a whole, and it simplifies to a Schur polynomial corresponding to a staircase partition (see~\cite{Oka,Stroganov-IK-sym}).

The main goal of the present work is to prove an analogous result for the higher spin generalization 
of the six-vertex model with domain wall boundary conditions: by setting $q=\exp (2\pi i/(2\ell+1))$, the partition function reduces to a Macdonald polynomial corresponding to a staircase partition. As a consequence, the partition function is a symmetric polynomial in the full set of spectral parameters.

\subsection*{Outline of the paper}

The first two sections are introductory: 
Section~\ref{sec:def-6v} gives the definition of the $6$-vertex model with domain-wall boundary conditions and the explicit determinantal form of the partition function. In Section~\ref{sec:fusion} the higher spin generalization of the $6$-vertex model and the analogous representation of its partition function are presented.

The following two sections contain the original results of the paper:
Section~\ref{sec:Prop_Z} presents an alternative representation of the partition function in terms of determinants of scalar products of rational functions, and this reformulation is used to prove some basic but important properties of the partition function, valid for all values of $q$. In Section~\ref{sec:main} we prove that the partition function satisfies the wheel condition when $q=\exp(2\pi i/(2\ell+1))$, and, by using the results of ~\cite{FJMM-Macdonald}, we show that there is a well-defined unique Macdonald polynomial satisfying the same vanishing constraints. It is shown that, up to a multiplicative constant, there is a unique polynomial with the prescribed degrees and symmetries that satisfies the wheel condition, and therefore the partition function and the Macdonald polynomial coincide, up to an explicit multiplicative constant. The proof of the uniqueness lemma and the calculation of the proportionality constant are left to the appendices.

\section{Review of the $6$-vertex model}\label{sec:def-6v}

In this section we give a brief description of the $6$-vertex model, on a square grid, with domain wall boundary conditions, following the construction presented in~\cite{tese}.

\subsection{Definition of the model}

Take a square grid of size $n \times n$, in which each edge is given an orientation (an arrow), such that at each vertex there are two incoming and two outgoing arrows, which gives six possibilities. Alternatively, the arrows can be represented by signs according to the rule that arrows pointing right or upward correspond to plus signs and arrows pointing downward or left correspond to minus signs. Impose the \emph{domain wall boundary conditions} prescribing that the arrows at the top and bottom boundaries are outgoing and the ones at the left and right boundaries are incoming. See, e.g., Figure~\ref{fig:6v}.

\beginfig
\begin{center}
\includegraphics{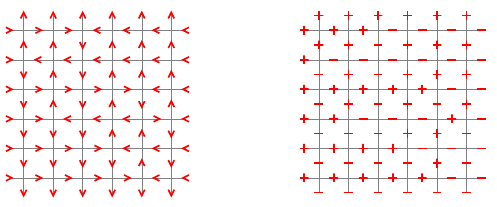}
\end{center}
\caption{A $6\times 6$ configuration of the $6$-vertex model in terms of arrows (left) or signs (right) \label{fig:6v}}
\end{figure}

To each vertex configuration we assign a weight
\begin{equation}
w(x,y)
=\begin{cases}
 a(x,y) =qx-q^{-1}y\\
 b(x,y) = x-y\\
 c(x,y) = (q-q^{-1})\sqrt{xy}\ ,
\end{cases}
\end{equation}
according to Figure~\ref{fig:poids}. The parameter $q$ is called the \emph{crossing parameter} of the model, while the parameters $x$ and $y$, called \emph{spectral parameters}, depend on the row and the column of the vertex, respectively. Let $\bm{x}=\{x_1,\ldots,x_n\}$ and $\bm{y}=\{y_1,\ldots,y_n\}$ be the horizontal and vertical spectral parameters, respectively.

\beginfig
\includegraphics{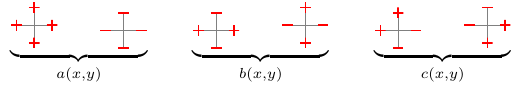}
\caption{Weights of vertex configurations \label{fig:poids}}
\end{figure}

The weight of a configuration is defined by the product of the weights of the vertices.
The partition function is defined as the sum of the weights over all possible configurations:
\begin{align}\label{eq:Z_sum}
 Z_n (\bm{x},\bm{y}) \coleq \sum_{\text{configurations}}\prod_{i,j=1}^n w_{ij} (x_i,q y_j)\ .
\end{align}

The partition function is renormalized in the following way:
\begin{equation}\label{eq:Z_renorm}
 \mathcal{Z}_n (\bm{x},\bm{y})= (-1)^{\binom{n}{2}} q^{- \nicefrac{n^2}{2}} (q-q^{-1})^{-n} \left(\prod_{i=1}^n x_i^{-\nicefrac{1}{2}}y_i^{-\nicefrac{1}{2}} \right)Z_n (\bm{x},\bm{y})\ .
\end{equation}
The function $\mathcal{Z}_n (\bm{x},\bm{y})$ is an homogeneous polynomial of total degree $n(n-1)$ and of partial degree $n-1$ in each variable $x_i$ or $y_i$.

\subsection{Integrability} Let $V$ be the standard representation of $sl_2$ spanned by the eigenvectors $|+\rangle$ and $|-\rangle$ of $S^z$. The key ingredient to the exact solvability of the $6$-vertex model is the \emph{R-matrix}
\begin{equation}
R(x,y):=
   \left(
   \begin{array}{cccc}
     a & 0 & 0 & 0 \\
     0 & b & c & 0 \\
     0 & c & b & 0 \\
     0 & 0 & 0 & a
   \end{array}
      \right)
\end{equation}
representing an endomorphism
\begin{align}
\nonumber
R \colon V \otimes V &\to V\otimes V\\
\label{eq:R-matrix}
\ket{\epsilon_1} \otimes \ket{\epsilon_2} &\mapsto R_{\epsilon_1 \epsilon_2}^{\epsilon_3 \epsilon_4} \ket{\epsilon_3} \otimes \ket{\epsilon_4} 
\end{align}
where $\epsilon_i \in \{+,-\}$. In what follows, we consider the vector space $\bigotimes_{k=1}^{n} V_k$, where each $V_k$ is a labelled copy of $V$. We use the abbreviated notation
\begin{equation}
|\epsilon_1 \epsilon_2\cdots\epsilon_{n}\rangle := |\epsilon_1\rangle\otimes |\epsilon_2\rangle\otimes \cdots \otimes |\epsilon_{n}\rangle\ , \qquad \epsilon_i = \pm \quad (i=1,2,\dots, n)
\end{equation}
for canonical basis in the tensor product representation.  The matrix $R_{ij}$ stands for the map that acts as $R$ on $V_i\otimes V_j$ and as identity elsewhere.

 The action of the matrix $R$ can be interpreted as completing an allowed sign configuration at a vertex with prescribed left and bottom signs, as shown below:
\begin{center}
\includegraphics{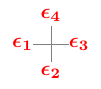}
\end{center}

The  $R$-matrix satisfies the \emph{Yang-Baxter equation}:
\begin{equation}
\label{eq:ybe}
 R_{23} (y_2,y_3) R_{13} (y_1,y_3) R_{12} (y_1,y_2)=
 R_{12} (y_1,y_2) R_{13} (y_1,y_3) R_{23} (y_2,y_3)\ ,
\end{equation}

and the \emph{inversion equation}:
\begin{align}
\label{eq:inversion}
R_{21}(y,x)R_{12}(x,y) = (qy-q^{-1}x)(qx-q^{-1}y) \Id\ . 
\end{align}

The \emph{transfer matrix} of the model is defined as
\begin{equation}
T(x,\bm{y}) = {}_{0}\langle - | R_{0n} (x,q y_n) \ldots R_{02} (x,q y_2) R_{01} (x,q y_1) 
 |+ \rangle_0\ ,
\end{equation}
where the matrix $R_{0i}$ acts on the tensor product of the $i^{\text{th}}$ space and the so-called auxiliary space $V_0$. In terms of the transfer matrix, the partition function is given by:
\begin{equation}\label{eq:Z_Tr_R}
 Z_n (\bm{x},\bm{y}) = \langle ++\cdots +
    |T(x_1,\bm{y}) T(x_2,\bm{y}) \ldots T(x_n,\bm{y}) |--\cdots -\rangle\ .
\end{equation}

Using the Yang--Baxter equation, the renormalized partition function $\mathcal Z_n(\bm x, \bm y)$, defined in Equation~\eqref{eq:Z_renorm}, is shown to be the so-called Korepin--Izergin determinant~\cite{Kor, Iz-6V}
\begin{align}\label{eq:IK_det}
 \mathcal{Z}_n (\bm{x},\bm{y}) = \frac{\prod_{i,j} (x_i-q y_j)(x_i-q^{-1}y_j)}{\Delta(\bm x)\Delta(\bm y)} \det \left| \frac{1}{(x_i-q y_j)(x_i-q^{-1}y_j)}\right|_{i,j=1}^n\ ,
\end{align}
where $ \Delta(\bm x)=\prod_{1\leq i<j\leq n} (x_i-x_j)$.
\subsection{Combinatorial point}
When $q=\exp(2\pi i/3)$, the Korepin--Izergin determinant~\eqref{eq:IK_det} dramatically simplifies and becomes a Schur polynomial~\cite{Oka,Stroganov-IK-sym}:
\begin{equation}
 \mathcal{Z}_n (\bm{x},\bm{y}) = s_{\delta_n} (\bm{x},\bm{y})\ ,
\end{equation}
where $\delta_n=(n-1,n-1,n-2,\ldots,2,1,1,0,0)$.
It follows that the partition function $\mathcal{Z}_n (\bm x, \bm y)$, at the combinatorial point, is a fully symmetric polynomial in the $2n$ variables $\{\bm x, \bm y\}$.

\section{Fusion and the higher spin generalization of the $6$-vertex model}
\label{sec:fusion}

We introduce the higher spin generalization of the $6$-vertex model considered in this paper. The corresponding $R$-matrix is constructed using fusion techniques for the representations of $sl_2$, as briefly explained below  (see~\cite{Fusion-Reshetikhin}). 
\subsection{Fusion and the generalized $R$-matrix}
The representation $\SymV$  is the irreducible component of $V^{\otimes \ell}$ spanned by the vectors
\begin{equation}
 |\ell;\ell-m\rangle \coleq \frac{1}{m!} \left(S^-\right)^m |\underbrace{+ +\ldots + }_{\ell} \rangle\qquad m=0,1,\dots, \ell
\end{equation}
(see~\cite{FH}). The $R$-matrix \eqref{eq:R-matrix} can be used to build an endomorphism of the vector space
\begin{equation}
\label{eq:decomp}
V^{\otimes \ell} \otimes V^{\otimes \ell} = V_1 \otimes \ldots \otimes V_{\ell} \otimes V_{\ell+1} \otimes \ldots \otimes V_{2\ell}\ ,
\end{equation}
with spectral parameters
\begin{equation}
\label{eq:spectral_p}
\{x,q^2 x,\ldots,q^{2\ell-2}x,y,q^2 y,\ldots,q^{2\ell-2}y\}
\end{equation}
associated to the $2\ell$ factors of the tensor product decomposition \eqref{eq:decomp}.
\begin{defi} The operator
\begin{equation}
R^{(\ell)} (x,y) \ \colon \ V^{\otimes \ell} \otimes V^{\otimes \ell} \to V^{\otimes \ell} \otimes V^{\otimes \ell}\ 
\end{equation}
is defined by
\begin{align}
\label{eq:r_l_def}
 R^{(\ell)} (x,y) & \coleq R_{1,2\ell}(x,q^{2\ell-2}y) R_{1,2\ell-1}(x,q^{2\ell-4}y) \ldots R_{1,\ell+1} (x,y)\\
\nonumber
 & \qquad \times R_{2,2\ell} (q^2 x,q^{2\ell-2}y) R_{2,2\ell-1}(q^2 x,q^{2\ell-4}y) \ldots R_{2,\ell+1} (q^2 x,y)\\
 \nonumber
 & \qquad \ldots \\ 
\nonumber
 & \qquad \times R_{\ell,2\ell} (q^{2\ell-2} x,q^{2\ell-2}y) R_{\ell,2\ell-1}(q^{2\ell-2} x,q^{2\ell-4}y) \ldots R_{\ell,\ell+1} (q^{2\ell-2} x,y)\ .
\end{align}
\end{defi}

Given that $R$ satisfies the Yang--Baxter equation~\eqref{eq:ybe} and the inversion relation~\eqref{eq:inversion}, it can be shown that so does $R^{(\ell)}$:
\begin{prop}
The matrix $R^{(\ell)}(x,y)$ satisfies the Yang--Baxter equation
\begin{equation}
\label{eq:ybe_l}
 R^{(\ell)}_{23} (y_2,y_3) R^{(\ell)}_{13} (y_1,y_3) R^{(\ell)}_{12} (y_1,y_2)=
 R^{(\ell)}_{12} (y_1,y_2) R^{(\ell)}_{13} (y_1,y_3) R^{(\ell)}_{23} (y_2,y_3)
\end{equation}
and the inversion equation 
\begin{equation}
\label{eq:inversion_l}
 R^{(\ell)} (y,x) R^{(\ell)} (x,y) \propto \Id \ .
\end{equation}
\end{prop}

Note that~\eqref{eq:ybe_l} and~\eqref{eq:inversion_l} hold with general spectral parameters; the special choice~\eqref{eq:spectral_p} allows the operator $R^{(\ell)}$ to be restricted to the subspace $\SymV\otimes \SymV$, as shown below.

\begin{prop} \label{prop:sym_iff}
 A state $|v\rangle \in V^{\otimes \ell}$ belongs to $\SymV$ if and only if
\begin{equation}
 R_{i,i+1} (q^{2i-2} x, q^{2i} x) |v\rangle = 0
\end{equation}
for all $1\leq i\leq \ell-1$.
\end{prop}

\begin{proof}
It is enough to show that the operator $R_{i,i+1} (q^{2i-2} x, q^{2i}x)$ annihilates a state $|v\rangle$ if and only if $|v\rangle$ is invariant under exchanging positions $i$ and $i+1$, since the $R$-matrix is a local operator. The $R$-matrix for the special choice of spectral parameters $(q^{2i-2}x,q^{2i}x)$ reads as

\begin{equation}
 R_{i,i+1}(q^{2i-2}x,q^{2i}x) = (q^2-1) q^{2i-2}x
   \left(
   \begin{array}{cccc}
     0 & 0 & 0 & 0 \\
     0 & -1 & 1 & 0 \\
     0 & 1 & -1 & 0 \\
     0 & 0 & 0 & 0
   \end{array}
      \right)\ ,
\end{equation}
whose kernel is exactly $\mathrm{Sym}^{2}V\subset V\otimes V$.
\end{proof}

\begin{lemma}
 The operator $R^{(\ell)}$ leaves $\SymV\otimes \SymV$ invariant and therefore the map
\begin{equation}
 R^{(\ell)} (x,y): \SymV\otimes \SymV \to \SymV\otimes \SymV
\end{equation}
is well-defined.
\end{lemma}

\begin{proof}
Let $|v\rangle \otimes |w\rangle \in \SymV\otimes \SymV$. Note that the commutation relation
\begin{multline}
\label{eq:comm}
 R_{i,i+1}(q^{2i-2}x,q^{2i}x) \left(R^{(\ell)}(x,y) |v\rangle \otimes |w\rangle \right)\\=
 R^{(\ell)}(x,y) \left(R_{i,i+1} (q^{2i-2}x,q^{2i}x)|v\rangle\right) \otimes |w\rangle
\end{multline}
holds for $1\leq i\leq \ell-1$, as a consequence of the Yang--Baxter equation \eqref{eq:ybe} applied to \eqref{eq:r_l_def}. 

Since $|v\rangle$ belongs to $\SymV$, the r.h.s.~of~\eqref{eq:comm} vanishes for $1\leq i\leq \ell-1$. Therefore, by Proposition \ref{prop:sym_iff}, the vector $R^{(\ell)}(x,y) |v\rangle \otimes |w\rangle$ is symmetric in its first $\ell$ factors. A similar argument shows that $R^{(\ell)}(x,y) |v\rangle \otimes |w\rangle$ is also symmetric in its last $\ell$ factors, and hence $R^{(\ell)}(x,y) |v\rangle \otimes |w\rangle \in \SymV \otimes \SymV$.
\end{proof}

\subsection{Higher spin model}

Take a $n\times n$ grid as in the $6$-vertex model above, and to each edge associate an integer $0 \leq \alpha \leq \ell$, which labels the corresponding state $|\ell; \alpha \rangle \in \SymV$. We assume \emph{generalized domain wall boundary conditions}, that is, $\ell$ is assigned to the edges on the left and top boundaries and $0$ is fixed along the edges of the bottom and right boundaries. As in the $6$-vertex model, $2n$ spectral parameters $\{\bm x,\bm y\}$ are associated to the vertical and horizontal lines of the grid (see Figure \ref{fig:boundary_sp_p}).
\beginfig
\label{fig:boundary_sp_p}
\includegraphics{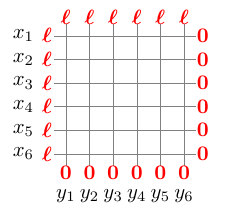}
\caption{Boundary conditions and spectral parameters for the $6\times 6$ grid}
\end{figure}

Analogously to the $6$-vertex model, the following conservation condition is imposed on a vertex configuration:
\begin{center}
\includegraphics{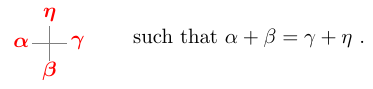}
\end{center}
By using the standard notation
\begin{equation}
R^{(\ell)} (x_i,y_j)|\ell;\alpha\rangle \otimes |\ell;\beta\rangle = \left.R^{(\ell)}\right._{\alpha,\beta}^{\gamma,\eta}(x_i,y_j) |\ell;\gamma\rangle \otimes |\ell;\eta\rangle\ ,
\end{equation}
the weight of a vertex configuration is given by
\begin{equation}
\label{eq:l_spin_weights}
 w_{i,j}(x_i,y_j) = \left.R^{(\ell)}\right._{\alpha,\beta}^{\gamma,\eta}(x_i,qy_j)\ .
\end{equation}
This choice guarantees the integrability of the model.

\subsection{Partition function}
The partition function $Z_{n,\ell} (\bm x, \bm y)$ of the spin $\ell$ model is defined exactly as in~\eqref{eq:Z_sum}, where now the weights $w_{ij}$ are those given in \eqref{eq:l_spin_weights}.
The fusion process, defining  $R^{(\ell)}$ from $R$, allows to express the partition function of the spin $\ell$ model in terms of the original $6$-vertex model partition function as
\begin{equation}
\label{eq:l_spin_vs_old}
 Z_{n,\ell} (\bm x,\bm y) = Z_{\ell n} (\bm{\bar x},\bm{\bar y})\ ,
\end{equation}
where
\begin{equation}
\bm{\bar{x}} = \{x_1,q^2 x_1,\ldots,q^{2\ell-2}x_1,\ldots,x_n,q^2 x_n,\ldots,q^{2\ell-2}x_n\}\ ,
\end{equation}
and $\bm{\bar{y}}$ is defined similarly. 

As in \eqref{eq:Z_renorm}, the function
\begin{equation}
\mathcal{\hat Z}_{n,\ell}(\bm x,\bm y)=\prod_i^n \left(x_i y_i\right)^{-\nicefrac{\ell}{2}}Z_{n,l}(\bm x, \bm y)
\end{equation}
is an homogeneous polynomial in the variables $x_i$ and $y_i$. Moreover, $\mathcal{\hat Z}_{n,\ell}(\bm x,\bm y)$ is divisible by the product
\begin{equation}
{\prod_{i,j}^n \prod_{p=0}^{\ell-2} \prod_{k=0}^{\ell-1} (q^{2k}x_i-q^{2p+1}y_j)}\ ,
\end{equation}
as a consequence of the Korepin--Izergin formula \eqref{eq:IK_det} evaluated at $\bm{\bar x}$ and $\bm{\bar y}$.
The reduced partition function is defined as
\begin{equation}\label{eq:redef_Znl}
 \mathcal{Z}_{n,\ell}(\bm{x},\bm{y})\coleq \frac{\prod_{i,j=1}^n \prod_{p=0}^{\ell} \prod_{k=0}^{\ell-1} (q^{2k}x_i-q^{2p-1}y_j)}{\Delta(\bm{\bar{x}})\Delta(\bm{\bar{y}})} \det \mathcal{A}_{\ell} (\bm{x},\bm{y}) 
\end{equation}
where $\mathcal{A}_{\ell} (\bm{x},\bm{y})$ is the $\ell n\times \ell n$ matrix given by:
\begin{align}
 \mathcal{A}_{\ell} (\bm{x},\bm{y}) &\coleq \left[A_{\ell} (x_{\alpha},y_{\beta})\right]_{\alpha,\beta=1}^n
\end{align}
with $\ell\times \ell$ blocks of the form
\begin{align}
 A_{\ell} (x,y) &\coleq \left[\frac{1}{(q^{2j} x-q^{2i-1} y)(q^{2j}x-q^{2i+1} y)}\right]_{i,j=0}^{\ell-1}\ .
\end{align}
Note that we recover the Korepin--Izergin determinant when $\ell =1$.

The original partition function can be written as
\begin{equation}
Z_{n,\ell}(\bm x,\bm y) = \text{Const.} \prod_i^n \left(x_i y_i\right)^{\nicefrac{\ell}{2}}\prod_{i,j}^n \prod_{p=0}^{\ell-2} \prod_{k=0}^{\ell-1} (q^{2k}x_i-q^{2p+1}y_j) \mathcal{Z}_{n,\ell} (\bm x,\bm y)\ ,
\end{equation}
where the constant can be determined explicitly.

\section{Alternative representation of the partition function}\label{sec:Prop_Z}
For rational functions $u(z)$ and $v(z)$ the \emph{residue pairing} is defined as
\begin{equation}
 \left\langle u(z),v(z) \right\rangle = - \res_{z=\infty} u(z)v(z)dz \ ,
\end{equation}
in terms of which the Korepin--Izergin formula \eqref{eq:IK_det} can be presented as
\begin{equation}
\label{eq:IK_scalar_pr_1}
\mathcal{Z}_n(\bm x,\bm y) = \frac{1}{\Delta(\bm x)\Delta(\bm y)}\det\left(\left\langle p_i(\bm x; z),\frac{1}{z-y_j} \right\rangle\right)_{i,j=1}^{n}\ ,
\end{equation}
where
\begin{equation}
 p_{i} (\bm x;z) = \prod_{k=1\atop k\neq i}^{n} (z-qx_k)(z-q^{-1}x_k)\ , \qquad i=1,\dots, n\ .
\end{equation}
In what follows, we use the following alternative representation of the partition function:
\begin{prop}
\label{prop:IK_scalar_pr}
The standard $6$-vertex model partition function $\mathcal{Z}_n(\bm x,\bm y)$ can be written as
\begin{equation}
\label{eq:IK_scalar_pr_2}
\mathcal{Z}_n(\bm x,\bm y) = \det\left(\left\langle r_i(\bm x;z),\frac{z^{j-1}}{w(\bm y; z)} \right\rangle\right)_{i,j=1}^{n}\ ,
\end{equation}
where
\begin{equation}
w(\bm y;z) =\prod_{k=1}^{n}(z-y_k)\ , 
\end{equation}
and 
\begin{equation}
r_{i} (\bm x;z) = \frac{w(\bm x;qz)(q^{-1}z)^{i-1} - w(\bm x; q^{-1}z)(qz)^{i-1}}{(q-q^{-1})z}\ , \qquad i=1,\dots, n\ .
\end{equation}
\end{prop}
\begin{proof} First observe that if $\Delta(\bm y)\not=0$ then
\begin{equation}
\sp\left\{\frac{1}{z-y_j}\right\}_{j=1}^{n} = \sp\left\{\frac{z^{j-1}}{w(\bm y;z)}\right\}_{j=1}^{n}\ ,
\end{equation}
and the change of basis is given explicitly as
\begin{equation}
\frac{z^{i-1}}{w(\bm y;z)}= \sum_{i=1}^{n}M_{ij}(\bm y)\frac{1}{z-y_j}\ , \qquad i=1,\dots, n\ ,
\end{equation}
where $M$ is the $n\times n$ matrix
\begin{equation}
\label{eq:matrix_m}
M_{ij}(\bm y) =  \frac{y_j^{i-1}}{\prod_{k \not=j}(y_j-y_k)}\ , \qquad i,j=1,\dots n\ .
\end{equation}
Similarly, if $\Delta(\bm x)\not=0$ then
\begin{equation}
\sp\left\{p_j(\bm x;z)\right\}_{j=1}^{n} = \sp\left\{r_j(\bm x;z)\right\}_{j=1}^{n}\ ,
\end{equation}
and the change of basis can be written, using again \eqref{eq:matrix_m}, as
\begin{equation}
r_i(\bm x;z) = \sum_{j=1}^{n}M_{ij}(\bm x)p_j(\bm x,z)\ .
\end{equation}
To conclude the proof, it is enough to note that
\begin{equation}
\det(M(\bm x)) = \frac{(-1)^{\binom{n}{2}}}{\Delta(\bm x)}\ .
\end{equation}
\end{proof}
\begin{lemma}
The higher spin partition function is given by
\begin{equation}
\label{eq:higher_r_det}
\mathcal{Z}_{n,\ell}(\bm x,\bm y) = \det\left(\left\langle \tilde r_i(\bm x; z),\frac{z^{j-1}}{w(\bar{\bm y};z)} \right\rangle\right)_{i,j=1}^{n\ell}\ ,
\end{equation}
where
\begin{equation}
\tilde r_i(\bm x;z) = q^{n(\ell-1)}\frac{w(\bm x;qz)(q^{-1}z)^{i-1} - w(\bm x; q^{-2\ell+1}z)(qz)^{i-1}}{(q-q^{-1})z}\ ,
\end{equation}
for $1\leq i \leq n\ell$\ .
\end{lemma}
\begin{proof}
Define the polynomial
\begin{equation}
\pi(\bm x;z) =q^{n(\ell-1)^2}\prod_{k=1}^{\ell-1}w(\bm x;q^{-2k+1}z)\ .
\end{equation}
Observe that
\begin{align}
w(\bar{\bm x};qz) &=q^{n(\ell-1)}\pi(\bm x; z)w(\bm x; qz)\\
w(\bar{\bm x};q^{-1}z) &=q^{n(\ell-1)}\pi(\bm x; z)w(\bm x; q^{-2\ell+1}z)\ ,
\end{align}
and hence
\begin{equation}
r_i(\bar{\bm x};z) = \pi(\bm x; z) \tilde r_i (\bm x; z)\ , \qquad i=1,\dots, n\ell\ .
\end{equation}
The identity \eqref{eq:higher_r_det} follows by factoring out the product
\begin{equation}
\prod_{j=1}^{n}\prod_{k=0}^{\ell-1}\pi(x;q^{2k}y_j)
\end{equation}
of the l.h.s. of \eqref{eq:IK_scalar_pr_2}.
\end{proof}

\begin{prop}
\label{prop:Z_l_properties}
The partition function $\mathcal{Z}_{n,\ell} (\bm{x},\bm{y})$
satisfies the following properties:

\begin{enumerate}
\item\label{item:homogen}  $\mathcal{Z}_{n,\ell} (\bm{x},\bm{y})$ is an homogeneous polynomial in the set of variables $\{\bm{x},\bm{y}\}$,
\item  $\mathcal{Z}_{n,\ell}(\bm{x},\bm{y})$ is  symmetric  in the variables $\bm{x}$ and in the variables $\bm{y}$,
\item $\mathcal Z_{n,\ell} (\bm x, \bm y) = \mathcal Z_{n, \ell} (\bm y, \bm x)$,
\item $\mathcal{Z}_{n,\ell}(\bm{x},\bm{y})$ has total degree at most $\ell n(n-1)$,
\item $\mathcal{Z}_{n,\ell}(\bm{x},\bm{y})$ has partial degree at most $\ell(n-1)$ in each variable $x_i$ or $y_i$.
\end{enumerate}
\end{prop}

\begin{proof}
The exchange symmetry (iii) follows easily from the representation \eqref{eq:redef_Znl}. The determinantal representation \eqref{eq:higher_r_det} shows that $Z_n(\bm x,\bm y)$ is symmetric homogeneous polynomial in $\bm x$, and therefore in $\bm y$. The total degree of the partition function can be read off from \eqref{eq:redef_Znl}.

To compute the partial degree of $\mathcal Z_{n, \ell} (\bm x, \bm y)$ in the variable $y_1$, note that \eqref{eq:higher_r_det} can be slightly modified as
\begin{equation}
\label{eq:higher_r_det_m}
\mathcal{Z}_{n,\ell}(\bm x,\bm y) = \frac{(-1)^{\binom{n\ell}{2}}}{\Delta(\bar{\bm y})}\det\left(\left\langle \tilde r_i(\bm x; z),\frac{1}{z-\bar{y}_j} \right\rangle\right)_{i,j=1}^{n\ell}\ ,
\end{equation}
where the only the first $\ell$ columns of the determinant depend on $y_1$. The degree of the polynomial $\tilde r_i(\bm x;z)$ in $z$ is equal to $n+i-2$ and therefore the $i$th row in the first $n\ell \times \ell$ block of the matrix in
\eqref{eq:higher_r_det_m} consists of polynomials of partial degree $n+i-2$ in the variable $y_1$. This means that the highest possible exponent of $y_1$ appearing in the determinant \eqref{eq:higher_r_det_m} is equal to $\sum_{k=1}^{\ell}(n+n\ell-k-1)=n\ell(\ell+1)-\frac{1}{2}\ell(\ell+3)$. To conclude (v), it is enough to recall that the partial degree of $\Delta(\bar{\bm y})$ in the variable $y_1$ is $\frac{1}{2}\ell(\ell-1)+(n-1)\ell^2$.
\end{proof}
\begin{rmk} In general, the partition function $\mathcal{Z}_{n,\ell} (\bm{x},\bm{y})$ is not fully symmetric in the set of $2n$ variables $\{\bm{x},\bm{y}\}$.
\end{rmk}
\section{The combinatorial point}\label{sec:main}

In this section we present our main result: at the combinatorial point, $q = \exp(2 \pi i /(2 \ell+1))$, the partition function $\mathcal Z_{n,\ell} (\bm x, \bm y)$ is a certain Macdonald polynomial.

\subsection{The main result}

Let $\ell \delta_n$ be the staircase partition with $n$ steps $2 \times \ell$, that is $\ell \delta_n = (\ell(n-1),\ell( n-1), \ldots,\ell,\ell,0,0)$.
Notice that the total degree and the partial degree of the partition function $\mathcal Z_{n,\ell} (\bm x, 	\bm y)$ are equal to, respectively, $|\ell \delta_n| = \ell n (n-1)$ and the first part of $\ell \delta_n$. 
Let
\begin{equation}
  \rho_\ell = e^{\nicefrac{2\pi i}{(2\ell+1)}}\ .
\end{equation}

\begin{thm}\label{thm:main}
At the combinatorial point $q=\rho_\ell$, the partition function $\mathcal Z_{n, \ell} (\bm x, \bm y)$ is a Macdonald polynomial ~\cite{Macdonald}, up to a multiplicative constant, more precisely
\begin{equation}
 \mathcal{Z}_{n,\ell} (\bm{x},\bm{y}) = \gamma_{n,\ell} P_{\ell \delta_n} (\bm{x},\bm{y};\rho_\ell^2,\rho_\ell)\ .
\end{equation}
\end{thm}

The proportionality constant $\gamma_{n,\ell}$ will be given explicitly in Proposition~\ref{prop:propto}.

\begin{rmk}
 The coefficients $u_{\lambda \mu}(q,t)$ in the expansion of the Macdonald polynomial  
 \begin{equation}\label{eq:Macdonald_expansion}
   P_\lambda (\bm z; q, t) = \sum_{\mu \leq \lambda} u_{\lambda\mu}(q,t) m_{\mu} (\bm z)
 \end{equation}
 are rational functions of $q$ and $t$ and therefore the specialization $q = \rho_\ell^2$ and $t = \rho_\ell$ has to be done carefully. This issue is addressed in Subsection~\ref{sec:wheel_macdonald}.
\end{rmk}

Theorem~\ref{thm:main} has the following important consequence:
\begin{cor}
 At the combinatorial point $q=\rho_\ell$, the partition function $\mathcal{Z}_{n,\ell} (\bm{x},\bm{y})$ is a fully symmetric polynomial.
\end{cor}

In order to prove Theorem~\ref{thm:main}, we show below that the partition function $\mathcal Z_{n, \ell} (\bm x, \bm y)$ satisfies a set of constraints, known as the wheel condition, see Theorem~\ref{thm:wheel_condition}.
The wheel condition together with the properties described in Proposition~\ref{prop:Z_l_properties} is sufficient to characterize the partition function up to a multiplicative constant, see Lemma~\ref{lemma:uniqueness}.
Using a result of Feigin et al, we prove that there is a unique Macdonald polynomial of total degree $\ell n(n-1)$ which also satisfies the wheel condition, see Proposition~\ref{prop:FJMM}. Therefore the uniqueness result of Lemma~\ref{lemma:uniqueness} implies  Theorem~\ref{thm:main}.

\subsection{The wheel condition}
Let $\bm z = \{z_1, \ldots, z_{2n}\}$.
\begin{defi}[Wheel condition]
 Let $q$ and $t$ be such that $q^{r-1} t^{k+1} = 1$, for some non-negative integers $k$ and $r$.
 A function $f(\bm{z})$ is said to obey the $(r,k)_{q,t}$-wheel condition if $f(\bm z)$ vanishes whenever
\begin{equation}
 \frac{z_{i_{\alpha+1}}}{z_{i_{\alpha}}}= t q^{s_{\alpha}} \quad \text{for any}\  s_{\alpha}\in \mathbb{N} \ \text{such that}\ \sum_{\alpha=1}^k s_{\alpha}\leq r-1\ ,
\end{equation}
and for any choice of $1\leq i_1 < i_2 < \ldots < i_{k+1} \leq 2n$.
\end{defi}

 \begin{lemma}\label{lemma:wheel_condition_xxx}
  Let $n \geq 3$. At the combinatorial point $q = \rho_\ell$, the partition function $\mathcal Z_{n, \ell} (\bm x, \bm y)$ vanishes whenever
   \begin{equation}
     x_3 = q^{1+2s_2} x_2 = q^{2+2s_1+2s_2} x_1
   \end{equation}
  for any $s_1$, $s_2 \in \mathbb N$ such that $s_1+s_2 \leq \ell-1$.
 \end{lemma}

\begin{proof}
 Let $\mathcal S = \sp\left\{r_i (\bm {\bar x}; z)\right\}_{i=1}^{\ell n}$. 
 The polynomial 
 \begin{multline}
   a(\bm{\bar x}; z) = \left(\prod_{i=4}^n \prod_{k = 0}^{\ell-1} (z-q^{2 k}x_i)\right) \left( \prod_{k=0}^{s_1-1} (z-q^{2k} x_1)\right)\\
   \left( \prod_{k=s_1}^{s_1+s_2-1} (z-q^{1+2k}x_1) \right) \left(\prod_{k=s_1+s_2}^{\ell-2} (z-q^{2+2k}x_1) \right)\ ,
\end{multline}
 of degree $\ell(n-2) -1$, is such that 
\begin{equation}
  \frac{w(\bm{\bar x}; qz) a(\bm{\bar x}; q^{-1}z)-w(\bm{\bar x}, q^{-1}z)a(\bm{\bar x}; qz)}{(q-q^{-1})z} = 0\ .
\end{equation}
Thus $\dim \mathcal S < \ell n$.
 This implies that the matrix appearing in Equation~\eqref{eq:higher_r_det} is singular.
\end{proof}

\medskip
\noindent

 \begin{lemma}\label{lemma:wheel_condition_xxy}
 Let $n \geq 2$. At the combinatorial point $q = \rho_\ell$,
  the partition function $\mathcal Z_{n, \ell} (\bm x, \bm y)$ vanishes whenever
   \begin{equation}
     y_1 = q^{1+2s_2} x_2 = q^{2+2s_1+2s_2} x_1
   \end{equation}
  for any $s_1$, $s_2 \in \mathbb N$ such that $s_1+s_2 \leq \ell-1$.
 \end{lemma}

 \begin{proof}
 Notice that $\bar y_{\ell-s_1-s_2+j} = q^{2j-1} x_1$ for $0 \leq j\leq s_1$, because $y_1 = q^{2+2s_1+2s_2}x_1$. 
 Let $\tilde{\mathcal S} = \sp\left\{\tilde r_i (\bm x; z)\right\}_{i=1}^{\ell n}$ and let 
 \begin{equation}
   a_i (x; z) = z^{i-1}\prod_{j=0}^{s_1-1} (z-q^{2j} x)\ ,
 \end{equation}
 for $1 \leq i \leq \ell n - s_1$.
 Then the polynomials
 \begin{equation}
   r^\prime_i (\bm{\bar x};z) = \frac{w(\bm x;qz)a_i (x_1; q^{-1}z)-w(\bm x;q^{-2\ell+1}z)a_i (x_1; qz)}{(q-q^{-1})z}\ , 
 \end{equation}
 belonging to $\tilde{\mathcal S}$, vanish at $z= q^{2j-1} x_1$ for $0 \leq j \leq s_1$.

 Let $A_{ij} = \left\langle \tilde r_i (\bm x; z), \frac{1}{z-\bar y_j} \right\rangle$, then the partition function is given by
 \begin{equation}
  \mathcal{Z}_{n,\ell}(\bm x,\bm y) = \frac{(-1)^{\binom{\ell n}{2}}}{\Delta (\bm{\bar y})}\det\left( A_{ij}\right)_{i,j=1}^{n\ell}\ .
 \end{equation}
 Let $\mathcal S^{\prime} = \sp\{r_i^\prime (\bm x;z)\}_{i=1}^{\ell n-s_1}$ be a subspace of $\tilde{\mathcal S}$, and let $A_{ij}^\prime = \left\langle r_i^\prime (\bm x; z), \frac{1}{z-\bar y_j} \right\rangle$.
 The entries $A^\prime_{ij}$ vanish when $\ell-s_1-s_2 \leq j\leq \ell -s_2$, then $A^\prime$ is of rank at most $\ell n -s_1-1$.
 Therefore $A$ is of rank at most $\ell n -1$.
 \end{proof}

 \begin{lemma}\label{lemma:wheel_condition_xyy}
  Let $n \geq 2$. At the combinatorial point $q = \rho_\ell$, the partition function $\mathcal Z_{n, \ell} (\bm x, \bm y)$ vanishes whenever
   \begin{equation}
     y_2 = q^{1+2s_2} y_1 = q^{2+2s_2+2s_2} x_1
   \end{equation}
  for any $s_1$, $s_2 \in \mathbb N$ such that $s_1+s_2 \leq \ell-1$.
 \end{lemma}

 \begin{proof}
  This follows from Lemma~\ref{lemma:wheel_condition_xxy}, by using the symmetries of the partition function.
 \end{proof}

 \begin{lemma}\label{lemma:wheel_condition_yyy}
  Let $n \geq 3$. At the combinatorial point $q = \rho_\ell$, the partition function $\mathcal Z_{n, \ell} (\bm x, \bm y)$ vanishes whenever
   \begin{equation}
     y_3 = q^{1+2s_2} y_2 = q^{2+2s_1+2s_2} y_1
   \end{equation}
  for any $s_1$, $s_2 \in \mathbb N$ such that $s_1+s_2 \leq \ell-1$.
 \end{lemma}

 \begin{proof}
  Recall that $\mathcal Z_{n,\ell} (\bm x, \bm y) = \mathcal Z_{n,\ell} (\bm y, \bm x)$, and hence this follows from Lemma~\ref{lemma:wheel_condition_xxx}.  
 \end{proof}
By summarizing the above, we obtain the following crucial property of the partition function at the combinatorial point:
\begin{thm}\label{thm:wheel_condition}
 At the combinatorial point $q = \rho_\ell$, the partition function $\mathcal{Z}_{n,\ell} (\bm{x},\bm{y})$ satisfies the $(\ell,2)_{\rho_\ell^2,\rho_\ell}$-wheel condition.
\end{thm}

\begin{defi}\label{def:Vn}
 Let $\bm x = \{x_1, \ldots, x_n \}$ and $\bm y = \{ y_1,\ldots, y_n\}$. 
 The vector space $V_n$ is defined as the space of polynomials $p (\bm x, \bm y)$ such that
\begin{enumerate}
 \item $p (\bm x, \bm y)$ is an homogeneous polynomial in the set of variables $\{\bm x,\bm y\}$,
 \item $p (\bm x, \bm y)$ is symmetric in the variables $\bm x$, and also in the variables $\bm{y}$,
 \item $p(\bm x, \bm y) = p (\bm y, \bm x)$,
 \item $p (\bm x, \bm y)$ has total degree at most $\ell n (n-1)$,
 \item $p (\bm x, \bm y)$ has partial degree at most $\ell (n-1)$ in each variable $x_i$ or $y_i$,
 \item $p(\bm x, \bm y)$ satisfies the $(\ell,2)_{\rho_\ell^2, \rho_\ell}$-wheel condition.
\end{enumerate}
\end{defi}

\begin{lemma}\label{lemma:uniqueness}
 The vector space $V_n$ is at most one dimensional.
\end{lemma}

The proof of the above lemma is given in Appendix~\ref{sec:uniqueness}.

\subsection{The wheel condition and Macdonald polynomials}\label{sec:wheel_macdonald}

\begin{defi}[Admissible partitions]
A partition $\lambda=(\lambda_1,\lambda_2,\ldots)$ is said to be $(r,k)$-admissible if and only if $\lambda_i - \lambda_{i+k} \geq r$ for all $i$.
\end{defi}

According to the result of Feigin et al~\cite{FJMM-Macdonald}, the space of symmetric polynomials satisfying the wheel condition is spanned by Macdonald polynomials.
More precisely:
\begin{thm}[\cite{FJMM-Macdonald}]\label{thm:FJMM}
 Let $q$ and $t$ be two generic scalars such that $q^{r-1}t^{k+1}=1$, and let $\mathcal{V}$ denote the space of symmetric polynomials in $\bm z$ satisfying the $(r,k)_{q,t}$-wheel condition.
Let $\mathcal{M}$ be the space spanned by the Macdonald polynomials $P_{\lambda} (\bm z; q, t)$ indexed by $(r,k)$-admissible partitions.
 Then $\mathcal{V}=\mathcal{M}$.
\end{thm}

In the case of relevance for our purposes this simplifies to
\begin{prop}\label{prop:FJMM}
 Let $q$ and $t$ be two generic scalars such that $q^{\ell-1}t^{3} = 1$, and let $p (\bm{z})$ be an homogeneous symmetric polynomial of total degree $\ell n (n-1)$ satisfying the $(\ell,2)_{q,t}$-wheel condition.
Then
\begin{equation}
  p(\bm z) \propto P_{\ell \delta_n} (\bm{z};q,t)\ .
\end{equation}
\end{prop}

\begin{proof}
The partition $\lambda = \ell \delta_n$ is the only $(\ell,2)$-admissible partition such that $|\lambda|= \ell n (n-1)$.
\end{proof}

Because we do not use generic $q$ and $t$ extreme care is required before we apply the theorem of Feigin et al. Let $m = \gcd(3, \ell-1)$.
The locus of the equation $q^{\ell-1}t^3=1$ splits into $m$ branches: $q^{\nicefrac{(\ell-1)}{m}}t^{\nicefrac{3}{m}}=\omega$, where $\omega$ is a $m$-th root of unity.
We consider the branch corresponding to $\omega = \exp(2 \pi i / m)$, parametrized by a variable $u$:
\begin{align}\label{eq:r_t_param}
 \begin{aligned}
  q(u)&=u^3 \\
  q(u)&=u 
 \end{aligned}
  &&
 \begin{aligned}
  t(u)&=u^{-(\ell-1)} \\
  t(u)&=u^{-\frac{\ell-1}{3}} e^{\frac{2\pi i}{3}}
 \end{aligned}
   &&
 \begin{aligned}
  \text{for } \ell &\equiv 0,2 \quad \text{(mod 3)}\\
  \text{for } \ell &\equiv 1 \quad \text{(mod 3)}\ .
\end{aligned}
\end{align}
Let $u_0$ be such that $q(u_0) = \rho_\ell^2$ and $t(u_0) = \rho_\ell$. 
Explicitly
\begin{align}\label{def:u0}
 \begin{aligned}
  u_0 &=e^{\frac{4\pi i}{3(2\ell+1)}}e^{\frac{2\pi i}{3}}\\
  u_0 &=e^{\frac{4\pi i}{2\ell+1}}\\
  u_0 &=e^{\frac{4\pi i}{3(2\ell+1)}} e^{-\frac{2\pi i}{3}}
 \end{aligned}
    &&
 \begin{aligned}
  \text{for } \ell &\equiv 0 \quad\text{(mod 3)}\\
  \text{for } \ell &\equiv 1 \quad\text{(mod  3)}\\
  \text{for } \ell &\equiv 2 \quad \text{(mod 3)}\ .
 \end{aligned}
\end{align}

In~\cite{FJMM-Macdonald}, it was shown that the coefficients $u_{\lambda \mu} (q(u),t(u))$ in the expansion~\eqref{eq:Macdonald_expansion} are well-defined rational functions in $u$.
When $u$ approaches $u_0$, the coefficients $u_{\ell \delta_n, \mu} (q(u),t(u))$ behave like $(u-u_0)^{n_{\mu}}$ for some power $n_{\mu}$.
Let $N = -\min_{\mu} \{n_{\mu}\}$, which is non-negative because $n_{\ell \delta_n} = 0$.
Then the renormalized polynomial 
\begin{equation}\label{eq:tilde_P}
\tilde P_{\ell \delta_n} (\bm z;q(u),t(u))=(u-u_0)^N P_{\ell \delta_n}(\bm z;q(u),t(u))
\end{equation}
is a regular function when $u$ approaches $u_0$.

Notice that $\tilde P_{\ell \delta_n} (\bm z; q(u),t(u))$ is an homogeneous symmetric polynomial of total degree $\ell n (n-1)$ satisfying the wheel condition when $u \neq u_0$.
 The limit $\lim_{u \to u_0} \tilde P_{\ell \delta_n} (\bm z;q(u),t(u))$ is well-defined, and it is simply given by:
\begin{equation}
  \lim_{u \to u_0} \tilde P_{\ell \delta_n}(\bm z;q(u),t(u)) = \tilde P_{\ell \delta_n} (\bm z; \rho_\ell^2,\rho_\ell)\ .
\end{equation}

\begin{lemma}
 The polynomial $\tilde P_{\ell \delta_n} (\bm z; \rho_\ell^2, \rho_\ell)$ satisfies the $(\ell,2)_{\rho_\ell^2, \rho_\ell}$-wheel condition.
\end{lemma}

\begin{proof}
 This lemma follows from the regularity of $\tilde P_{\ell \delta_n} (\bm z; q(u), t(u))$ when $u$ approaches $u_0$.
\end{proof}

\subsection{The proof of the main theorem}
Let
\begin{align}
 [\ell]_q! &= [\ell]_q [\ell-1]_q \ldots [1]_q\ , & \text{where}\ [a]_q &=\frac{q^a-1}{q-1}\ .
\end{align}
\begin{prop}\label{prop:propto}
 The constant $\gamma_{n,\ell}$, defined in Theorem~\ref{thm:main}, is given by:
\begin{equation}
 \gamma_{n,\ell} = \left((-1)^{\binom{\ell}{2}} \rho_\ell^{2\binom{\ell}{2}(2n-3)}[\ell]_{\rho_\ell^2}!\right)^n\ .
\end{equation}
\end{prop}

In order to prove this proposition, in Appendix~\ref{app:leading} we compute the coefficient of $\prod_{i=1}^n (x_i y_i)^{\ell (i-1)}$ in the partition function.
Since this coefficient is non-zero, the Macdonald polynomial 
$P_{\ell \delta_n} (\bm x, \bm y; \rho_{\ell}^2, \rho_{\ell})$ is well-defined 
and 
\begin{equation}
P_{\ell \delta_n}(\bm x, \bm y; \rho_{\ell}^2, \rho_{\ell}) =\tilde  P_{\ell \delta_n} (\bm x, \bm y; \rho_{\ell}^2, \rho_{\ell})\ .
\end{equation}
\begin{proof}[Proof of Theorem~\ref{thm:main}]
 Note that the definition of $V_n$ is such that $\mathcal{Z}_{n,\ell}(\bm{x},\bm{y}) \in V_n$, and $P_{\ell \delta_n} (\bm{x},\bm{y};\rho_\ell^2, \rho_\ell)\in V_n$. 
 But by Lemma~\ref{lemma:uniqueness}, $V_n$ is at most one dimensional, and therefore $\mathcal{Z}_{n,\ell} (\bm{x},\bm{y}) = \gamma_{n,\ell}P_{\ell \delta_n}(\bm{x},\bm{y};\rho_\ell^2, \rho_\ell)$, with coefficient $\gamma_{n,\ell}$ computed in Proposition~\ref{prop:propto}.
\end{proof}

\section{Final remarks}

\subsection{Combinatorial interpretation}

A canonical higher spin generalization of an alternating sign matrix (ASM), as defined by Behrend and Knight~\cite{Higher-ASM}, is a square matrix with entries in $\{-\ell, \ldots, -1, 0, 1, \ldots, \ell\}$, such that sum of all entries in each row or column is $\ell$, and the partial sum of the first or last $r$ entries in each row or column is non-negative for each $r$. The bijection between such matrices and the configurations of the model considered in this paper can be established similarly to the case $\ell=1$:
reading each column of the grid from the bottom to the top, if the value of the edge goes from $\beta$ to $\eta$, the corresponding entry of the higher spin ASM is $\eta-\beta$.  Alternatively, the constraint \eqref{eq:higher_config_constraint} on a configuration guarantees that the same result is obtained by reading each row from right to left. Unfortunately, the weights considered in this paper seem very unnatural in the combinatorial setting.\medskip

By setting all variables of the partition function $\mathcal Z_{n}(\bm x,\bm y)$ to be equal to $1$ at the combinatorial point $q=\exp(2\pi i/3)$, we get the famous sequence
\begin{equation}
1, 1, 2, 7, 42, 429, 7436, 218348, ... \quad \mbox{(sequence A005130 in OEIS)}
\end{equation}
that counts several objects, including alternating sign matrices and totally symmetric self-complementary plane partitions. We have seen that the combinatorial point $q=\exp((2\ell+1)/2\pi i)$, the higher spin partition function $\mathcal Z_{n,\ell}(\bm x,\bm y)$ is a Macdonald polynomial. Therefore, it would be interesting to find some combinatorial interpretation of the values obtained for $\ell > 1$ in the homogeneous limit when $x_i = y_i =1$ for all $i$.

\subsection{The wheel condition}

The wheel condition for the higher spin partition function is deduced using the special determinantal form $\mathcal Z_{n,\ell}(\bm x,\bm y)$ that involves rational function entries with a specific dependence of the crossing parameter $q$. This motivates a more systematical study of the connection between the wheel condition and more general determinants of a similar type. See, for example \cite{Lascoux_Gaudin}, for a different generalization of the Korepin--Izergin determinant that satisfies a wheel condition.
\medskip

There are other integrable statistical models that are of interest from this point of view, such as, for example, the $so_n$ model in~\cite{DowFoda} and the eight-vertex model in \cite{eight_vertex_1,eight_vertex_2}.

\subsection{Symmetry}

It is very intriguing that we start with a function that exhibits an $S_n \times S_n$ symmetry but not $S_{2n}$ symmetry in general. It is natural to ask
what the special features of this determinant are that imply the extra symmetry at $q=\exp((2\ell+1)/2\pi i)$. It would be expected that the symmetry comes from the physics of the model, that is, there exists some mechanism (like the Yang--Baxter equation) which, for this very special value of $q$, allows to exchange a row with a column. Another idea to explain the symmetry might come from  Stroganov's article~\cite{Stroganov-IK-sym}.

\subsection{Relation to KP $\tau$-functions}
The alternative representation \eqref{eq:higher_r_det} of the partition function
has a Grassmannian manifold interpretation, and this point of view is intimately connected with the fact that the $6$-vertex model with domain wall boundary conditions can be seen as a $\tau$-function of the KP hierarchy, where the spectral parameters $\bm{x}$ and $\bm{y}$ play the role of Miwa variables associated to the commuting flows of the hierarchy. 
This was first observed by Foda et al.~\cite{six_vertex_KP}, based on a fermionic vacuum expectation value representation of the partition function. 
In his survey paper~\cite{Takasaki_KP}, Takasaki extends this and other related results by showing how partition functions of certain 2D solvable models and scalar products of Bethe vectors from integrable spin chain models can also be written as KP $\tau$-functions.
\subsection*{Acknowledgements} The authors thank the Centre de Recherches Math\'emati\-ques in Montr\'eal where most of the research presented here was carried out.
The work of T.~F.~was partially supported by ANR project DIADEMS (Developing an Integrable Approach to Dynamical and Elliptic Models).
The work of F.~B.~was partially supported by the FP7 IRSES project RIMMP (Random and Integrable models in Mathematical Physics), the ERC project FroM-PDE (Frobenius Manifolds and Hamiltonian Partial Differential Equations) and the MIUR Research Project  Geometric and analytic theory of Hamiltonian systems in finite and infinite dimensions.

The authors are grateful to the referees for their constructive remarks to improve the structure and the readability of the paper.
\appendix

\section{Uniqueness} \label{sec:uniqueness}

The goal of this section is to prove Lemma~\ref{lemma:uniqueness}. 
That is, if there exists a nonzero polynomial $p (\bm x, \bm y) \in V_n$, then it is unique up to a multiplicative constant.
 
 Let $\bm z = \{z_1, \ldots, z_{2n}\} = \{ \bm x, \bm y\}$ and $q = \rho_\ell$. 
 We generalize the $(\ell,2)_{q^2, q}$-wheel condition:
\begin{defi}[$r$-wheel condition]
 Let $r$ be such that $3r < \ell$.
 A polynomial $p(\bm z)$ is said to obey the $r$-wheel condition if $p(\bm z)$ vanishes whenever
	\begin{equation}
   z_k = q^{1+2r+2s_2} z_j = q^{2+4r+2s_1+2s_2} z_i 
  \end{equation}
 for any $s_1$, $s_2 \in \mathbb N$ such that $s_1+ s_2 \leq \ell-1-3r$, and any choice of $1 \leq i < j < k \leq 2n$.
\end{defi}

 The following lemma holds:
\begin{lemma}\label{lemma:uniqueness_extra}
 Let $a (\bm x, \bm y) \in V_n$ be such that $\left.a (\bm x, \bm y)\right|_{y_j=qx_i} = 0$ for a given $i$ and $j$, then $a(\bm x, \bm y) = 0$.
\end{lemma}

\begin{proof}
 The result is trivial when $n=1$.
 Assume $n>1$.
 By symmetry, $a(\bm x, \bm y) \propto \prod_{i,j} (y_j - q x_i)(x_i - q y_j)$, therefore, there exists $b(\bm x, \bm y)$ such that 
\begin{equation}
 a(\bm x, \bm y) = \left(\prod_{i,j} (y_j - q x_i) (x_i-q y_j)\right) b(\bm x, \bm y)\ .
\end{equation}
The polynomial $b(\bm x, \bm y)$ has total degree at most $(\ell-2)n(n-1)-2n$ and partial degree at most $(\ell-2)(n-1)-2$ in each variable $x_i$ or $y_i$.
 When $\ell \leq 2$, $b(\bm x, \bm y) = 0$ and the lemma is proved.

 By the wheel condition, the following holds:
\begin{equation}
 \left.b (\bm{x},\bm{y})\right|_{x_j=q x_i} = \left(\prod_{k\neq i,j} \prod_{s=1}^{\ell}(x_k-q^{2s}x_i)\right) \left(\prod_k \prod_{s=2}^{\ell-1} (y_k - q^{2s}x_i)\right) b^\prime(\bm{x},\bm{y})\ .
\end{equation}
 The polynomial $b^\prime (\bm x, \bm y)$ has partial degree at most $-2n$ in $x_i$, and therefore it vanishes identically.
Thus
\begin{equation}
 a (\bm x, \bm y) = w_1 (\bm x, \bm y) a_1 (\bm x, \bm y)\ ,
\end{equation}
for some polynomial $a_1 (\bm x, \bm y)$ obeying the $1$-wheel condition, and 
\begin{equation}
 w_\alpha (\bm z) = \prod_{i \neq j} (z_i - q^\alpha z_j) \ .
\end{equation}
The polynomial $a_1 (\bm z)$ has total degree at most $(\ell-4)n(n-1)-2n$ and it has partial degree at most $(\ell-4)(n-1)-2$ in each variable $x_i$ or $y_i$.

By the $1$-wheel condition, the following holds:
\begin{equation}
 \left.a_1 (\bm z)\right|_{z_j=q^3 z_i} = \left(\prod_{k \neq i,j} \prod_{s=0}^{\ell-4} (z_k-q^{6+2s}z_i)\right) a^\prime_1 (\bm z)\ .
\end{equation}
The polynomial $a^\prime_1 (\bm z)$ has partial degree at most $-2(n-1) - 4$ in $z_i$, and therefore it vanishes identically.
Thus
\begin{equation}
 a_1 (\bm z) = w_3 (\bm z) a_2 (\bm{z})\ .
\end{equation}

This procedure can be iterated.
In the $r$-th step, we define $a_r (\bm z)$ by
\begin{equation}
 a_{r-1} (\bm z) = w_{2r-1} (\bm z) a_r (\bm z)\ .
\end{equation}
The polynomial $a_r (\bm z)$ obeys the $r$-wheel condition, has total degree at most $(\ell-4r)n(n-1) - 2rn$ and it has partial degree at most $(\ell-4r)(n-1) - 2r$ in each variable $z_i$.
The $r$-wheel condition implies:
\begin{equation}
 \left.a_r (\bm z)\right|_{z_j=q^{1+2r} z_i} = \left(\prod_{k \neq i,j} \prod_{s=0}^{\ell-1-3 r} (z_k-q^{2+4 r+2s}z_i)\right) a_r^\prime (\bm z)\ .
\end{equation}
The polynomial $a_r^\prime (\bm z)$ viewed as a function of $z_i$ has degree at most $-2r(n-1)-4r$, and therefore it vanishes identically.
This closes the iteration step.

We should stop the iteration at $r^* = \min \{ r \in \mathbb N\ \text{such that}\ 3r \geq \ell\}$.
The polynomial $a_{r^*} (\bm z)$ has negative total degree and therefore it vanishes identically.
\end{proof}

\begin{proof}[Proof of Lemma~\ref{lemma:uniqueness}]
 Use induction on $n$.
 Let $p_n (\bm x, \bm y) \in V_n$.

 The lemma holds for $n=1$. 
 By the wheel condition the following holds
 \begin{align*}
 \left.p_n (\bm x, \bm y) \right|_{y_j=qx_i} &= \left(\prod_{k\neq i} \prod_{s=1}^{\ell} (x_k-q^{2s}x_i)\right) \left(\prod_{k\neq j} \prod_{s=1}^{\ell}(y_k-q^{2s}x_i)\right) \hat p_n (\bm x, \bm y)\ .
 \end{align*}
 The polynomial $\hat p_n (\bm x, \bm y)$ does not depend either on $x_i$ or on $y_j$. It can then be checked that $\hat p_n (\bm z)\in V_{n-1}$, which by hypothesis is one-dimensional.

 For any nonzero polynomial $p_n^\prime (\bm x, \bm y) \in V_n$ there is a constant $\alpha$ such that 
\begin{equation}
\left.\left(p_n (\bm x, \bm y) - \alpha\, p_n^\prime (\bm x, \bm y)\right)\right|_{y_j=qx_i} = 0\ .
\end{equation}
Apply Lemma~\ref{lemma:uniqueness_extra}.
\end{proof}

\section{The coefficient $\gamma_{n,\ell}$} \label{app:leading}
\begin{prop}
The recursion identity
\begin{equation}
\label{eq:z_recursion}
\left. \mathcal{Z}_{n,\ell} (\bm x,\bm y)\right|_{x_n=y_n=0}
=(-1)^{\binom{\ell}{2}}q^{4\binom{\ell}{2}(n-1)-\binom{\ell}{2}}[\ell]_{q^2}!  \left(\prod_{k=1}^{n-1}x_k^{\ell} y_k^{\ell}\right) \mathcal{Z}_{n-1,\ell} (\hat{\bm x},\hat{\bm y})
\end{equation}
holds, where $\hat{\bm x} = (x_1,\dots, x_{n-1})$ and $\hat{\bm y} = (y_1,\dots, y_{n-1})$.
\end{prop}
\begin{proof} The determinant representation \eqref{eq:higher_r_det} of $ \mathcal{Z}_{n,\ell} (\bm x,\bm y)$ implies that
\begin{equation}
\left. \mathcal{Z}_{n,\ell} (\bm x,\bm y)\right|_{x_n=y_n=0}
=\det\left(\left\langle \left.\tilde r_i(\bm x; z)\right|_{x_n=0},\frac{z^{j-1}}{w(\bar{\hat{\bm y}};z)z^{\ell}} \right\rangle\right)_{i,j=1}^{n\ell}\ .
\end{equation}
Observe that
\begin{equation}
 \left. \tilde r_{\ell+i} (\bm x; z) \right|_{x_n=0} = z^{\ell+1} \tilde r_{i} (\hat{\bm x};z)\qquad i\geq 1\ .
\end{equation}
Note also, assuming $y_i\not=0$ for $1\leq i\leq n-1$, that
\begin{equation}
\sp\left\{\frac{z^{j-1}}{w(\bar{\hat{\bm y}};z)z^{\ell}}\right\}_{j=1}^{n\ell} = \sp\left(\left\{\frac{1}{z^{\ell-j+1}}\right\}_{j=1}^{\ell}\cup\left\{\frac{z^{j-1}}{w(\bm y;z)}\right\}_{j=1}^{(n-1)\ell}\right)\ ,
\end{equation}
with the change of basis
\begin{equation}
\label{eq:matrix_M}
\frac{z^{i-1}}{w(\bar{\hat{{\bm y}}};z)z^{\ell}}= \sum_{j=1}^{\ell}M_{i, j}(\hat{\bm y})\frac{1}{z^{\ell-j+1}}+\sum_{j=1}^{(n-1)\ell}M_{i,\ell+j}(\hat{\bm y})\frac{z^{j-1}}{w(\bar{\hat{{\bm y}}};z)}\ ,
\end{equation}
where the matrix $M(\hat{\bm y})$ defined by \eqref{eq:matrix_M} is upper triangular, and its determinant is equal to
\begin{equation}
\det(M(\hat{\bm y})) = \frac{1}{w(\bar{\hat{\bm y}};0)^\ell}=(-1)^{(n-1)\ell}\frac{q^{-(n-1)\ell^2(\ell-1)}}{\prod_{k=1}^{n-1}y_k^{\ell^2}}\ .
\end{equation}
A simple calculation gives that
\begin{multline}
\left\langle \left.\tilde r_i(\bm x; z)\right|_{x_n=0},\frac{1}{z^{\ell-j+1}} \right\rangle\\
=\left\{
\begin{array}{cc}
0 & i > \ell-j+1\\
\displaystyle(-1)^{n-1}q^{n(\ell-1)+i-2\ell+1}[\ell-i+1]_{q^2}\prod_{k=1}^{n-1}x_k& i=\ell-j+1
\end{array}
\right.\ .
\end{multline}
By combining the above, we obtain
\begin{align}
\left. \mathcal{Z}_{n,\ell} (\bm x,\bm y)\right|_{x_n=y_n=0}&=
\det(M(\hat{\bm y}))\det\left(\prod_{j=1}^{n}\left\langle \left.\tilde r_i(\bm x; z)\right|_{x_n=0},\frac{1}{z^{\ell-j+1}} \right\rangle\right)_{i,j=1}^{\ell} \times\\
&\quad \times \det\left(\left\langle z^{\ell+1}\tilde r_i(\hat{\bm x}; z),\frac{z^{j-1}}{w(\bar{\hat{\bm y}};z)} \right\rangle\right)_{i,j=1}^{(n-1)\ell}\ .
\end{align}
A short calculation gives
\begin{multline}
\det\left(\prod_{j=1}^{n}\left\langle \left.\tilde r_i(\bm x; z)\right|_{x_n=0},\frac{1}{z^{\ell-j+1}} \right\rangle\right)_{i,j=1}^{\ell}\\
= (-1)^{\binom{\ell}{2}+(n-1)\ell}q^{n(\ell-1)\ell-\frac{3\ell(\ell-1)}{2}}[\ell]_{q^2}! \left(\prod_{k=1}^{n-1}x_k\right)^{\ell}\ ,
\end{multline}
and we also have
\begin{multline}
\det\left(\left\langle z^{\ell+1}\tilde r_i(\hat{\bm x}; z),\frac{z^{j-1}}{w(\bar{\hat{\bm y}};z)} \right\rangle\right)_{i,j=1}^{(n-1)\ell}\\ = q^{(\ell-1)\ell(\ell+1)(n-1)}\left(\prod_{k=1}^{n-1}y_k^{\ell(\ell+1)}\right)\mathcal{Z}_{n-1,\ell} (\hat{\bm x},\hat{\bm y})\ ,
\end{multline}
from which the recursion formula \eqref{eq:z_recursion} follows.
\end{proof}

\begin{proof}[Proof of Prop.~\ref{prop:propto}]
Note that
\begin{equation}
\tilde r_i(x;z) = q^{i-\ell}[\ell-i+1]_{q^2}z^{i-1} +\mbox{lower order terms,}
\end{equation}
and hence
\begin{align}
\mathcal{Z}_{1,\ell} (x,y)&=q^{-\binom{\ell}{2}}[\ell]_{q^2}!\det\left(\left\langle z^{i-1},\frac{z^{j-1}}{w(\bar{y};z)} \right\rangle\right)_{i,j=1}^{n\ell}\\
&=(-1)^{\binom{\ell}{2}}q^{-\binom{\ell}{2}}[\ell]_{q^2}!\ .
\end{align}

The relation \eqref{eq:z_recursion} can be used recursively to compute the coefficient of the leading term $\prod_{i=1}^n (x_i y_i)^{\ell (i-1)}$ in $\mathcal{Z}_{n,\ell} (\bm x,\bm y)$ which gives
\begin{equation}
\left((-1)^{\binom{\ell}{2}} q^{\binom{\ell}{2}(2n-3)}[\ell]_{q^2}!\right)^n\ ,
\end{equation}
that coincides with the proportionality constant $\gamma_{n,\ell}$ since, by definition, the Macdonald polynomial $P_{\ell\delta_n}(\bm x,\bm y; q,t)$ has leading coefficient $1$.
\end{proof}

\bibliography{DN}

\def\polhk#1{\setbox0=\hbox{#1}{\ooalign{\hidewidth
  \lower1.5ex\hbox{`}\hidewidth\crcr\unhbox0}}}
\providecommand{\bysame}{\leavevmode\hbox to3em{\hrulefill}\thinspace}
\begin{thebibliography}{10}

\bibitem{Higher-ASM}
R.~E. Behrend and V.~A. Knight, \emph{Higher spin alternating sign matrices},
  Electron. J. Combin. \textbf{14} (2007), Research Paper 83, 38 pp,
  \url{http://www.combinatorics.org/ojs/index.php/eljc/article/view/v14i1r83}.
  \MR{MR2365982}

\bibitem{Bressoud}
D.~Bressoud, \emph{Proofs and confirmations: The story of the alternating sign
  matrix conjecture}, MAA Spectrum, Mathematical Association of America,
  Washington, DC, 1999. \MR{MR1718370}

\bibitem{CFKitanine-6V}
A~Caradoc, O~Foda, and N~Kitanine, \emph{Higher spin vertex models with domain
  wall boundary conditions}, Journal of Statistical Mechanics: Theory and
  Experiment \textbf{2006} (2006), no.~03, P03012,
  \href{http://arxiv.org/abs/math-ph/0601061}{\path{arXiv:math-ph/0601061}},
  \href{http://dx.doi.org/10.1088/1742-5468/2006/03/P03012}{\path{doi}}.
  \MR{MR2214730}

\bibitem{dG-review}
J.~de~Gier, \emph{Loops, matchings and alternating-sign matrices}, Discrete
  Math. \textbf{298} (2005), no.~1-3, 365--388,
  \href{http://arxiv.org/abs/math/0211285}{\path{arXiv:math/0211285}},
  \href{http://dx.doi.org/10.1016/j.disc.2003.11.060}{\path{doi}}.
  \MR{MR2163456}

\bibitem{DowFoda}
A~Dow and O~Foda, \emph{On the domain wall partition functions of level-1
  affine so(n) vertex models}, Journal of Statistical Mechanics: Theory and
  Experiment \textbf{2006} (2006), no.~05, P05010,
  \href{http://arxiv.org/abs/math-ph/0604006}{\path{arXiv:math-ph/0604006}},
  \href{http://dx.doi.org/10.1088/1742-5468/2006/05/P05010}{\path{doi}}.
  \MR{MR2231664}

\bibitem{FJMM-Macdonald}
B.~Feigin, M.~Jimbo, T.~Miwa, and E.~Mukhin, \emph{Symmetric polynomials
  vanishing on the shifted diagonals and {M}acdonald polynomials},
  International Mathematics Research Notices \textbf{2003} (2003), no.~18,
  1015--1034,
  \href{http://arxiv.org/abs/math/0209042}{\path{arXiv:math/0209042}},
  \href{http://dx.doi.org/10.1155/S1073792803209119}{\path{doi}}.
  \MR{MR1962014}

\bibitem{six_vertex_KP}
O.~Foda, M.~Wheeler, and M.~Zuparic, \emph{Domain wall partition functions and
  {KP}}, J. Stat. Mech. Theory Exp. (2009), no.~3, P03017, 20,
  \href{http://arxiv.org/abs/0901.2251}{\path{arXiv:0901.2251}},
  \href{http://dx.doi.org/10.1088/1742-5468/2009/03/P03017}{\path{doi}}.
  \MR{MR2495862}

\bibitem{tese}
T.~Fonseca, \emph{Alternating sign matrices, completely packed loops and plane
  partitions}, Ph.D. thesis, Université Pierre et Marie Curie, 2010,
  \url{http://tel.archives-ouvertes.fr/tel-00521884/fr/}.

\bibitem{FH}
William Fulton and Joe Harris, \emph{Representation theory}, Graduate Texts in
  Mathematics, vol. 129, Springer-Verlag, New York, 1991, A first course,
  Readings in Mathematics. \MR{1153249}

\bibitem{Iz-6V}
A.~G. Izergin, \emph{Partition function of a six-vertex model in a finite
  volume}, Dokl. Akad. Nauk SSSR \textbf{297} (1987), no.~2, 331--333.
  \MR{MR919260 }

\bibitem{Kor}
V.~Korepin, \emph{Calculation of norms of {B}ethe wave functions}, Comm. Math.
  Phys. \textbf{86} (1982), no.~3, 391--418,
  \url{http://www.springerlink.com/content/x6j8w6j351673l32}. \MR{MR677006 }

\bibitem{Fusion-KRS}
P.~Kulish, N.~Reshetikhin, and E.~Sklyanin, \emph{Yang-{B}axter equation and
  representation theory: I}, Letters in Mathematical Physics \textbf{5} (1981),
  393--403, \href{http://dx.doi.org/10.1007/BF02285311}{\path{doi}}.
  \MR{MR0649704}

\bibitem{Kup-ASM}
G.~Kuperberg, \emph{Another proof of the alternating-sign matrix conjecture},
  Internat. Math. Res. Notices (1996), no.~3, 139--150,
  \href{http://arxiv.org/abs/math/9712207}{\path{arXiv:math/9712207}},
  \href{http://dx.doi.org/10.1155/S1073792896000128}{\path{doi}}.
  \MR{MR1383754}

\bibitem{Lascoux_Gaudin}
A.~Lascoux, \emph{{G}audin functions, and {E}uler-{P}oincar\'e
  characteristics},
  \href{http://arxiv.org/abs/0709.1635}{\path{arXiv:0709.1635}}.

\bibitem{Macdonald}
I.~G. Macdonald, \emph{Symmetric functions and {H}all polynomials}, Oxford
  mathematical monographs, Oxford University Press Inc., 1979.

\bibitem{MRR-ASM}
W.~Mills, D.~Robbins, and H.~Rumsey, Jr., \emph{Alternating sign matrices and
  descending plane partitions}, J. Combin. Theory Ser. A \textbf{34} (1983),
  no.~3, 340--359. \MR{MR700040 }

\bibitem{MRR-TSSCPP}
W.~Mills, D.~Robbins, and H.~Rumsey, Jr, \emph{Self-complementary totally
  symmetric plane partitions}, J. Combin. Theory Ser. A \textbf{42} (1986),
  no.~2, 277--292. \MR{MR847558 }

\bibitem{Oka}
S.~Okada, \emph{Enumeration of symmetry classes of alternating sign matrices
  and characters of classical groups}, J. Algebraic Combin. \textbf{23} (2006),
  no.~1, 43--69,
  \href{http://arxiv.org/abs/math/0408234}{\path{arXiv:math/0408234}},
  \href{http://dx.doi.org/10.1007/s10801-006-6028-3}{\path{doi}}.
  \MR{MR2218849}

\bibitem{PM-19vertex}
R~A Pimenta and M~J Martins, \emph{The {Y}ang–{B}axter equation for {PT}
  invariant 19-vertex models}, Journal of Physics A: Mathematical and
  Theoretical \textbf{44} (2011), no.~8, 085205,
  \href{http://arxiv.org/abs/1010.1274}{\path{arXiv:1010.1274}},
  \href{http://dx.doi.org/10.1088/1751-8113/44/8/085205}{\path{doi}}.
  \MR{MR2770373}

\bibitem{Fusion-Reshetikhin}
N.~Reshetikhin, \emph{Lectures on the integrability of the 6-vertex model},
  Exact Methods in Low-dimensional Statistical Physics and Quantum Computing,
  vol.~89, 2010, Lecture Notes of the Les Houches Summer School,
  \href{http://arxiv.org/abs/1010.5031}{\path{arXiv:1010.5031}}.

\bibitem{eight_vertex_2}
Hjalmar Rosengren, \emph{An {I}zergin-{K}orepin-type identity for the 8{VSOS}
  model, with applications to alternating sign matrices}, Adv. in Appl. Math.
  \textbf{43} (2009), no.~2, 137--155,
  \href{http://dx.doi.org/10.1016/j.aam.2009.01.003}{\path{doi}}. \MR{2531917}

\bibitem{Stroganov-IK-sym}
Y.~G. Stroganov, \emph{A new way to deal with {I}zergin--{K}orepin determinant
  at root of unity}, 2002,
  \href{http://arxiv.org/abs/math-ph/0204042}{\path{arXiv:math-ph/0204042}}.

\bibitem{Takasaki_KP}
Kanehisa Takasaki, \emph{K{P} and {T}oda tau functions in {B}ethe ansatz}, New
  trends in quantum integrable systems, World Sci. Publ., Hackensack, NJ, 2011,
  pp.~373--391, \href{http://arxiv.org/abs/1003.3071}{\path{arXiv:1003.3071}}.
  \MR{MR2767954}

\bibitem{eight_vertex_1}
Wen-Li Yang and Yao-Zhong Zhang, \emph{Partition function of the eight-vertex
  model with domain wall boundary condition}, J. Math. Phys. \textbf{50}
  (2009), no.~8, 083518, 14,
  \href{http://dx.doi.org/10.1063/1.3205448}{\path{doi}}. \MR{2554446}

\end{thebibliography}
\bibliographystyle{amsplainhyper}
\end{document}